\documentclass[12pt]{amsart}
\usepackage[utf8]{inputenc}
\usepackage{graphicx} 
\usepackage{subcaption}
\usepackage{algorithm}
\usepackage{algpseudocode}
\usepackage{enumitem}
\usepackage{float}
\usepackage[left=1in, right=1in, top=1in, bottom=1in]{geometry}
\usepackage{colortbl}
\usepackage{times} 
\usepackage{wrapfig}       
\usepackage[table, dvipsnames]{xcolor}
\usepackage{amsmath,amssymb,amsthm,mathtools}
\usepackage{latexsym, verbatim}
\usepackage{appendix}
\usepackage{units}
\usepackage{listings}
\usepackage{amsmath}
\usepackage[normalem]{ulem}
\usepackage[hidelinks]{hyperref}
\usepackage{cleveref}
\usepackage{thm-restate}
\usepackage{fancyhdr}
\fancypagestyle{plain}{
    \fancyhf{} 
    \fancyfoot[C]{Page \thepage} 
}
\definecolor{darkred}{rgb}{0.8,0,0}
\definecolor{darkgreen}{rgb}{0,0.8,0}
\definecolor{darkblue}{rgb}{0,0,0.8}

\definecolor{lightred}{rgb}{1,0.8,0.8}
\definecolor{lightblue}{rgb}{0.8,0.8,1}
\definecolor{lightgreen}{rgb}{0.8,1,0.8}

\definecolor{codegreen}{rgb}{0,0.6,0}
\definecolor{codegray}{rgb}{0.5,0.5,0.5}
\definecolor{codepurple}{rgb}{0.58,0,0.82}
\definecolor{backcolour}{rgb}{0.95,0.95,0.95}

\lstdefinestyle{mystyle}{
    backgroundcolor=\color{backcolour},   
    commentstyle=\color{codegreen},
    keywordstyle=\color{magenta},
    numberstyle=\tiny\color{codegray},
    stringstyle=\color{codepurple},
    basicstyle=\ttfamily\tiny,
    breakatwhitespace=false,         
    breaklines=true,                 
    captionpos=b,                    
    keepspaces=true,                 
    numbers=left,                    
    numbersep=5pt,                  
    showspaces=false,                
    showstringspaces=false,
    showtabs=false,                  
    tabsize=2
}

\newtheorem{theorem}{Theorem}[section]
\newtheorem{corollary}[theorem]{Corollary} 
\newtheorem{lemma}[theorem]{Lemma}
\newtheorem{proposition}[theorem]{Proposition}

\usepackage{soul}
\usepackage{parskip}

\theoremstyle{definition}
\newtheorem{example}[theorem]{Example}

\theoremstyle{remark}
\newtheorem{remark}[theorem]{Remark}

\newcommand{\Lv}{\mathrm{Lv}}
\newcommand{\Int}{\mathrm{Int}}

\let\OLDthebibliography\thebibliography
\renewcommand\thebibliography[1]{
  \OLDthebibliography{#1}
  \setlength{\parskip}{0pt}
  \setlength{\itemsep}{0pt plus 0.3ex}
}

\title{Halfspace Representations of Path Polytopes of Trees}
\author{Amer Goel}
\address{Amer Goel, University of Michigan}
\email{amergoel@umich.edu}
\author{Aida Maraj}
\address{Aida Maraj, Max Planck Institute of Molecular Cell Biology and Genetics, and  Center for Systems Biology Dresden}
\email{maraj@mpi-cbg.de}
\author{\'Alvaro Ribot}
\address{\'Alvaro Ribot, Harvard University}
\email{aribotbarrado@g.harvard.edu}
\date{}

\usepackage{todonotes}
\usepackage{amsmath,amssymb}
\usepackage{dirtytalk}
\usepackage{tikz-network}
\usepackage{multirow}
\usepackage{bbm}

\def\V{\mathcal{V}}
\def\R{\mathbb{R}}

\def\Int{\mathrm{Int}}

\def\one{\mathbbm{1}}

\def\conv{\mathrm{conv}}

\def\leaf{\mathrm{leaf}}
\def\aff{\operatorname{aff}}
\def\ba{\mathbf{a}}
\def\bc{\mathbf{c}}
\def\bv{\mathbf{v}}
\def\bb{\mathbf{b}}
\def\bx{\mathbf{x}}
\def\by{\mathbf{y}}

\newcommand{\ojoin}{\mathbin{\ooalign{$\bigcirc$\cr\hidewidth$\vee$\hidewidth\cr}}}

\begin{document}
\begin{abstract}
Given a tree $T$, its path polytope is the convex hull of the edge indicator vectors for the paths between any two distinct leaves in $T$. These polytopes arise naturally in polyhedral geometry and  applications, such as phylogenetics, tropical geometry, and algebraic statistics. We provide a minimal halfspace representation of these polytopes. The construction is made inductively using toric fiber products.
\end{abstract}
\maketitle
\thispagestyle{plain}
\section{Introduction} \label{sec: intro}
A polytope is given either by its vertex representation ($\mathcal{V}$-representation) or halfspace representation ($\mathcal{H}$-representation). The $\mathcal{V}$-representation describes a polytope as the convex hull of its vertices and the $\mathcal{H}$-representation defines a polytope as the intersection of halfspaces. Therefore, a $\mathcal{V}$-representation is a parametric description of the polytope, and an $\mathcal{H}$-representation is an implicit description of the polytope.
We study \textit{path polytopes} of trees, which are defined by their $\mathcal{V}$-representation as follows: given two distinct leaves~$i,j$ in a tree $T$ with vertex set $V$ and edge set $E$, let $i \leftrightarrow j$ be the set of edges in the path in $T$ between $i$ and $j$, and let $\bc^{i \leftrightarrow j}\in \{0,1\}^{|E|}$  be the indicator vector for the edges used in $i \leftrightarrow j$.
The path polytope of the tree $T$, denoted $P_T$, is the convex hull of the vectors~$\bc^{i\leftrightarrow j}$ for any two distinct leaves $i,j$ in $T$. 
Our goal is to find its $\mathcal{H}$-representation, which provides a membership test for $P_T$. The Fourier-Motzkin elimination algorithm \cite{jing2019complexity} outputs the $\mathcal{H}$-representation given the vertices of the polytope, but its time complexity grows exponentially with the dimension of the polytope. Therefore, computing the $\mathcal{H}$-representation of path polytopes with this algorithm is infeasible for large trees. In \Cref{thm:main} we provide an explicit description of the $\mathcal{H}$-representation of any path polytope.

Path polytopes arrive naturally in graph theory, combinatorics, and  applications. In tropical geometry, phylogenetic trees are parametrized by the path in \cite{maclagan2021introduction}, related to the tropical Grassmannian \cite{speyer2004tropical}. A key motivation of this work is the growing recognition that many statistical models defined on trees or graphs are parametrized by paths between their nodes. Examples include Brownian motion tree models, where the parametrization is given by paths among leaves in a phylogenetic tree \cite{boege2021reciprocal};  staged tree models, parametrized by the paths from the root to a leaf on a rooted tree \cite{duarte2020equations}; and (colored) Gaussian graphical models on block graphs \cite{coons2023symmetrically,misra2021gaussian}, parametrized by paths between any two nodes, among others.

An explicit description of the halfspaces for $P_T$ in terms of the tree structure provides a better understanding of the polytope and is useful. For instance, polytopes arising from  monomial parametrizations of log-linear models, such as our polytope induced by the path parametrization, have shown to be useful in Maximum Likelihood Estimation (MLE) problems \cite{duarte2023toric,fienberg2012maximum}; given a normalized vector of counts for a log-linear model, the MLE exists if and only if this vector belongs to the relative interior of the corresponding polytope. Here, the halfspace description of the polytope gives a membership test. In~\cite{hollering2024hyperplane}, the authors use halfspace representations to learn causal polytree structures from a combination of observational and interventional data. In fact, the path parametrization has already shown essential for all the progress related to the MLE of Brownian motion tree models \cite{boege2021reciprocal,coons2024maximum}. Therefore, we anticipate that the halfspace representation of our path polytope will be valuable for statistical applications in the models discussed earlier. 

Before presenting our main result, let us introduce the necessary notation. A \emph{graph} is a tuple $G = (V, E)$ where $V$ is a set of nodes and $E$ is a set
of unordered pairs of nodes, which are called edges. We assume $1 < |V| < \infty$. A \emph{path} in a graph is a sequence of edges that connects two nodes.
A \emph{tree} is a graph in which any pair of nodes is connected by exactly one path.
Given a vertex $u$ in a tree $T=(V,E)$, let $N(u)=\{v \in V\mid \{u,v\}\in E\}$ be the neighborhood of~$u$. The degree of a vertex $u$ in $T$ is $\deg(u) = |N(u)|$. \emph{Leaves} are nodes of degree one, so connected to exactly one other node. Let $\Lv(T) \subset V$ denote the set of leaves of $T$, and $\Int(T) = V \setminus \Lv(T)$ denote the set of internal nodes in $T$. A \emph{star tree} is a tree $T$ with $|\Int(T)| = 1$. See \Cref{fig:example-graph-tree-path}.

\begin{figure}[h]
    \centering
    \begin{subfigure}{0.24\textwidth}
        \centering
        \begin{tikzpicture}[node distance=1cm, every node/.style={circle, draw, thick, fill=black, inner sep=2pt}]
            \node (1) {};
            \node (2) [right of=1] {};
            \node (3) [right of=2] {};
            \node (4) [below of=3] {};
            \node (5) [left of=4] {};
            \node (6) [left of=5] {};
            
            \draw[thick] (1) -- (2);
            \draw[thick] (2) -- (3);
            \draw[thick] (2) -- (4);
            \draw[thick] (3) -- (4);
            \draw[thick] (5) -- (6);
        \end{tikzpicture}
        \vspace{0.25cm}
        \caption{Graph $G$.}
        \label{fig: example-graph}
    \end{subfigure}
    \begin{subfigure}{0.24\textwidth}
        \centering
        \begin{tikzpicture}[node distance=1cm, every node/.style={circle, draw, thick, fill=black, inner sep=2pt}]
            \node (1) {};
            \node (2) [above left of=1] {};
            \node (3) [below left of=1] {};
            \node (4) [right of=1] {};
            \node (5) [above right of=4] {};
            \node (6) [below right of=4] {};
            
            \draw[thick] (1) -- (2);
            \draw[thick] (1) -- (3);
            \draw[thick] (1) -- (4);
            \draw[thick] (4) -- (5);
            \draw[thick] (4) -- (6);
        \end{tikzpicture}
        \caption{Tree $T$.}
        \label{fig:example-tree}
    \end{subfigure}
    \begin{subfigure}{0.24\textwidth}
        \centering
        \begin{tikzpicture}[node distance=1cm, every node/.style={circle, draw, thick, fill=black, inner sep=2pt}]
            \node (1) {};
            \node (2) [above left of=1, label=left:{$i$}] {};
            \node (3) [below left of=1] {};
            \node (4) [right of=1] {};
            \node (5) [above right of=4, label=right:{$j$}] {};
            \node (6) [below right of=4] {};
            
            \draw[dashed] (2) -- (1);
            \draw[dashed] (1) -- (4);
            \draw[thick] (1) -- (3);
            \draw[dashed] (4) -- (5);
            \draw[thick] (4) -- (6);
        \end{tikzpicture}
        \caption{Path $i \leftrightarrow j$ in $T$.}
        \label{fig: path}
    \end{subfigure}
    \begin{subfigure}{0.24\textwidth}
        \centering
        \begin{tikzpicture}[node distance=1cm, every node/.style={circle, draw, thick, fill=black, inner sep=2pt}]
            \node (1) {};
            \node (2) [above right of=1] {};
            \node (3) [below right of=1] {};
            \node (4) [above left of=1] {};
            \node (5) [below left of=1] {};
            
            \draw[thick] (1) -- (2);
            \draw[thick] (1) -- (3);
            \draw[thick] (1) -- (4);
            \draw[thick] (1) -- (5);
        \end{tikzpicture}
        \caption{Star tree $S_4$.}
        \label{fig:s3}
    \end{subfigure}
    \caption{Example of a graph, a tree, a path (dashed edges), and a star tree.}
    \label{fig:example-graph-tree-path}
\end{figure}
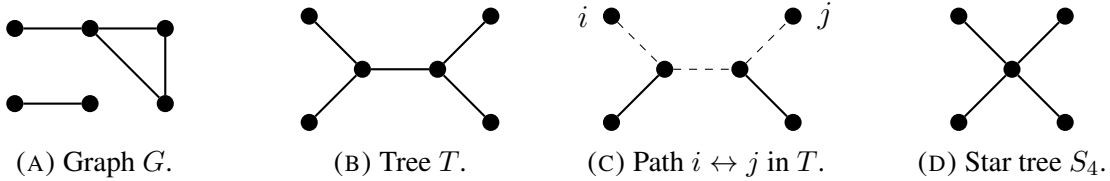
Let $E_{\leaf}(T) = \{ \{ u,v\} \in E \mid u \in \Lv(T) \text{ or } v \in \Lv(T)\}$ be the set of edges that have a leaf as an endpoint. Let $\R^E$ be the vector space with basis elements indexed by the set $E$, so $P_T \subset \R^E$. Our main result is the following.
\begin{theorem}
\label{thm:main}
    Given a tree $T=(V,E)$ with $|V|>3$ and no internal nodes of degree $2$, the path polytope $P_T$ has dimension $|E| - 1$ and a minimal $\mathcal{H}$-representation is given by
    \begin{align}\label{eq:main}
    \begin{cases}
        x_{\{u,v\}} &\geq 0  \text{\quad for all }  \{u,v\} \in E, \text{with }\deg(u),\deg(v) \neq 3\\
        -x_{\{u,v\}} + \sum\limits_{w \in N(u) \setminus \{v\}} x_{\{u,w\}} &\geq 0 \text{\quad for all } u \in \Int(T) \text{ and all } v \in N(u)\\
        \sum\limits_{\{u,v\} \in E_{\leaf}(T)} x_{\{u,v\}} &= 2.
    \end{cases}
\end{align}
\end{theorem}
\begin{remark}
    The condition that $T$ has no internal nodes of degree 2 is without loss of generality: we show that one can collapse all nodes of degree 2 to and get a new tree such that their path polytopes are isomorphic. By minimal $\mathcal{H}$-representation we mean that no halfspace is redundant. 
\end{remark}

\begin{remark}
    The internal nodes of degree $3$ are special for the following reason. We show that the polytope $P_T$ can be expressed as toric fiber products of pyramids over second hypersimplices. The second hypersimplex $\Delta_{n,2}$ has $2n$ facets for $n>3$ while $\Delta_{3,2}$ has only 3 facets. As such, internal nodes of degree 3 (which correspond to $\Delta_{3,2}$) induce fewer facets of $P_T$, and therefore fewer halfspaces.
\end{remark}

First, we check that every vertex in $P_T$ satisfies the conditions implied by the $\mathcal{H}$-representation given in \Cref{thm:main}, so $P_T$ is included in the polytope defined by that $\mathcal{H}$-representation. Then we prove that these halfspaces are sufficient to define $P_T$, and that none of these halfspaces is redundant. To do so, we express $P_T$ as the toric fiber product of polytopes coming from subtrees of $T$, allowing us to compute its facets inductively.

\textbf{Structure of the paper.} In \Cref{sec: preliminaries} we 
review literature on polytopes and toric fiber products that is relevant for this paper. In \Cref{sec: tfp} we show path polytopes are inductively constructed as toric fiber products of pyramids over path polytopes on star trees. Then, in \Cref{sec: facets for all trees}, we describe the facets and their respective halfspaces of path polytopes of trees, concluding the proof of 
\Cref{thm:main}.

\section{Preliminaries} \label{sec: preliminaries}

This section consists of three parts. First, we explain how a tree can be decomposed into a gluing of star trees, following the ideas in \cite{boege2021reciprocal}. Second, we recall the basic notions related to polytopes and define path polytopes of trees. Third, we recall toric fiber products of polytopes.

\subsection{Gluing of trees.} \label{subsec: gluing trees} 
Given two trees $T_1 = (V_1,E_1)$, $T_2 = (V_2,E_2)$, and two edges $e_1 \in E_\leaf(T_1)$ and~$e_2 \in E_\leaf(T_2)$, the \emph{gluing} of $T_1$ and $T_2$ along $e_1,e_2$ is the new tree  $T = T_1*_{e_1,e_2}T_2$ 
obtained as the union of~$T_1$ and $T_2$ by identifying $e_1\sim e_2$. That is, if $e_i = \{ v_i, k_i\}$ with $k_i \in \Lv(T_i)$ for $i=1,2$, then $V(T) = (V_1 \sqcup V_2) \setminus \{k_1, k_2 \}$ and $E(T) = (E_1 \cup E_2 \cup \{v_1, v_2 \}) \setminus \{e_1, e_2\}$. See  \Cref{fig:example-toric fiber product} for an example.

\begin{figure}[h]
    \begin{center}
    \begin{subfigure}{0.24\textwidth}
        \centering
        \begin{tikzpicture}[node distance=1cm, every node/.style={circle, draw, thick, fill=black, inner sep=2pt}]
            \node (1)[label = $1$] {};
            \node (2) [above left of=1, label = $2$] {};
            \node (3) [below left of=1, label = $3$] {};
            \node (4) [right of=1, label = $4$] {};
            
            \draw[thick] (1) -- (2);
            \draw[thick] (1) -- (3);
            \draw[thick] (1) -- (4);
        \end{tikzpicture}
        \caption{Tree $T_1$.}
    \end{subfigure}
    \raisebox{1.4cm}{$\ast_{\{1,4\},\{8,5\}}$}
    \begin{subfigure}{0.24\textwidth}
        \centering
        \begin{tikzpicture}[node distance=1cm, every node/.style={circle, draw, thick, fill=black, inner sep=2pt}]
            \node (1)[label = $5$] {};
            \node (2) [above right of=1, label = $6$] {};
            \node (3) [below right of=1, label = $7$] {};
            \node (4) [left of=1, label = $8$] {};
            
            \draw[thick] (1) -- (2);
            \draw[thick] (1) -- (3);
            \draw[thick] (1) -- (4);
        \end{tikzpicture}
        \caption{Tree $T_2$.}
    \end{subfigure}
    \raisebox{1.4cm}{$\mathbf{=}$}
    \begin{subfigure}{0.24\textwidth}
        \centering
        \begin{tikzpicture}[every node/.style={circle, draw, thick, fill=black, inner sep=2pt}]
            \node (1)[label = $1$] {};
            \node (2) [above left of=1, label = $2$] {};
            \node (3) [below left of=1,label = $3$] {};
            \node (4) [right of=1,label = $5$] {};
            \node (5) [above right of=4,label = $6$] {};
            \node (6) [below right of=4,label = $7$] {};
            
            \draw[thick] (1) -- (2);
            \draw[thick] (1) -- (3);
            \draw[thick] (1) -- (4);
            \draw[thick] (4) -- (5);
            \draw[thick] (4) -- (6);
        \end{tikzpicture}
        \caption{Tree $T$.}
    \end{subfigure}
    \end{center}
    \caption{Gluing $T_1$ and $T_2$ along edges $\{1,4\}$ and $\{8,5\}$ to form $T$.}
    \label{fig:example-toric fiber product}
\end{figure}
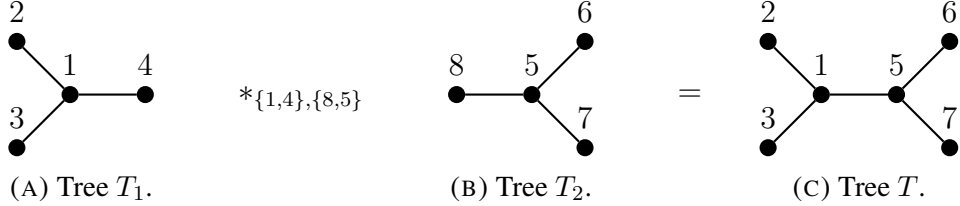

\begin{lemma}
    \label{lemma: trees decompose into stars}
   Let $T=(V,E)$ be a tree on $k$ internal nodes, with respective degrees $n_1,\ldots,n_k>2$. Then $T$ can be expressed as a gluing of $k$ star trees $S_{n_1},\ldots,S_{n_k}$. 
\end{lemma}
\begin{proof}
We argue by strong induction on $k$. If $k=1$ the statement is trivial, since a tree with one internal node is a star tree. Let $k>1$ and consider two internal nodes $u_1, u_2 \in V$ such that $\{u_1, u_2\} \in E$. Let $T_1$ be the subtree of $T$ consisting of all nodes $v$ in $T$ such that the path between~$v$ and $u_1$ does not pass through $u_2$, and define $T_2$ analogously. For $i=1,2$, let $T_i' = (V_i', E_i')$ where $V_i = V(T_i) \sqcup \{v_i\}$ and $E_i' = E(T_i) \sqcup \{\{u_i, v_i\}\}$, where $v_i$ is a newly introduced leaf of $T_i$ that connects to the internal node $u_i$. Then $T = T_1' *_{\{u_1, v_1\}, \{u_2,v_2\}}T_2'$ and $\Int(T_1'), \Int(T_2) < k$, so the claim follows by induction.
\end{proof}

\subsection{Path polytopes on trees} \label{subsec: polytopes of paths}
We follow the notation from \cite{ziegler2012lectures}.

A set $P \subset  \R^d$ is a \emph{polytope} if there exist points  $\bv_1,\ldots, \bv_k\in \mathbb R^d$ such that 
\[
P=\conv(\{\bv_1,\ldots, \bv_k\})= \left\{\sum\limits_{i=1}^k \lambda_i\bv_i \mid \lambda_i \geq 0, \sum\limits_{i=1}^k \lambda_i = 1\right\}.
\]
The set $ \V(P) = \{\bv_1,\ldots, \bv_k\}$ is called a $\mathcal V$- representation of $P$.
Given a vector $\ba \in \mathbb{R}^d$ and a scalar $b \in \mathbb{R}$, a linear inequality \( \ba^\top \bx \leq b \) is \textit{valid} for \( P \) if it is satisfied for all points \( \bx \in P \). A \textit{face} of \( P \) is any set of the form
\[
F = P \cap \{\bx \in \mathbb{R}^d \mid \ba^\top \bx = b\},
\]
where \( \ba^\top \bx \leq b \) is a valid inequality for \( P \). The \textit{dimension} of a face $F$ is the dimension of its affine hull $\operatorname{aff}(F)$, i.e., the smallest affine space containing $F$. We denote \( \dim(F) = \dim(\operatorname{aff}(F)) \). Similarly, the dimension of a polytope $P$ is $\dim(\operatorname{aff}(P))$.  A $d$-dimensional polytope is called a $d$-polytope.
Given a face $F = P \cap \{x \in \mathbb{R}^d \mid \ba^\top \bx = b\}$, the equation $\ba^\top \bx = b$ is called the \textit{supporting affine space} of $F$. A face $F$ of dimension $\dim(P)-1$ is called a \textit{facet} of $P$, and a corresponding valid inequality \( \ba^\top \bx \leq b \) is a \textit{halfspace} of~$P$. The set of affine spaces that contain~$P$ together with a set of halfspaces which describe the facets of $P$ is called an $\mathcal{H}$-representation of~$P$.  Given two positive integers $k, n$, the \emph{$(n,k)$-hypersimplex} is
\begin{equation*}\label{eq:hypersimplex}
    \Delta_{n,k} = \left\{ (x_1, \dots, x_n) \in \R^n \mid \sum_{i=1}^n x_i = k, \; 0 \leq x_1,\ldots,x_n \leq 1  \right\}.
\end{equation*}
When $k=1$, $\Delta_{n,1}$ is the \emph{standard simplex} $\Delta_n$. 

Consider a tree $T = (V, E)$. Recall that for two vertices $u,v\in V$, $u \leftrightarrow v$ denotes the set of edges in the path that connects $u$ and $v$, and the \textit{edge indicator vector} $\bc_T^{ u \leftrightarrow v} = (c_e)_{e \in E} \in \R^{E}$ has  $c_e = 1$ if~$e \in u \leftrightarrow v$ and $c_e = 0$ otherwise. For simplicity, we will use $\bc^{u \leftrightarrow v}$ instead of $\bc_T^{u \leftrightarrow v}$ if there is no risk of ambiguity. The \emph{path polytope} of the tree $T$ is
\[
P_{T} = \operatorname{\conv}(\{\bc_T^{i \leftrightarrow j} \mid i,j \in \Lv(T)\}) \subset \R^{E}.
\]

The next important notion for us is \textit{free join} of two polytopes, which refers to their union when they lie in \textit{skew} affine spaces, i.e., affine spaces that are neither parallel or intersecting. More precisely, let $P, Q$ be two polytopes such that $\dim (\conv(P \cup Q)) = \dim(P) + \dim(Q) + 1$. Then $\conv(P \cup Q)$ is called the \emph{free join} of $P$ and $Q$ and is denoted $P \ojoin Q$.
We are interested in taking the convex hull of path polytopes of trees with the origin. The next result shows that it is a free join, see \Cref{fig:plots_polytopes}.
\begin{proposition}
\label{prop: big hyperplane}
 Given a tree $T$ with at least two edges, the polytope $P_T$ lives in the hyperplane defined by $\sum_{e \in E_{\leaf}(T)} x_{e} = 2$. Consequently, $\conv({P}_{T} \cup \{\mathbf{0}\})$= $P_T \ojoin \{\mathbf{0}\}$.
\end{proposition}
\begin{proof}

Every vertex $\bc^{i \leftrightarrow j}$ of $P_T$ satisfies the equation $\sum_{e \in E_{\leaf}(T)} x_{e} = 2$.  Hence, $P_T$ lives in this hyperplane that does not contain the origin. This implies that $\conv(P_T \cup \{\mathbf{0}\})$ is the free join $P_T \ojoin \{\mathbf{0}\}$.
\end{proof}

\begin{figure}[ht!]
    \centering
    \begin{subfigure}{0.3\textwidth}
    \centering
    \vfill
        \begin{tikzpicture}[node distance=1cm, every node/.style={circle, draw, thick, fill=black, inner sep=2pt}]
            \node (1) []{};
            \node (2) [above right of=1] {};
            \node (3) [below right of=1] {};
            \node (4) [left of=1] {};
            \node[below of=3, draw=none, fill=none, circle=none] (n0) {};
            
            \draw[thick] (1) -- (2);
            \draw[thick] (1) -- (3);
            \draw[thick] (1) -- (4);
        \end{tikzpicture}
        \caption{Star tree $S_3$.}
        \vfill
    \end{subfigure}
    \begin{subfigure}{0.3\textwidth}
    \centering
    \includegraphics[width=\textwidth]{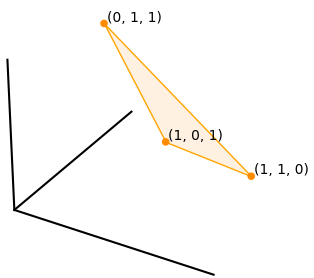}
    \caption{Path polytope $P_{S_3}$.}
    \end{subfigure}
  \qquad
    \begin{subfigure}{0.3\textwidth}
    \centering
    \includegraphics[width=\textwidth]{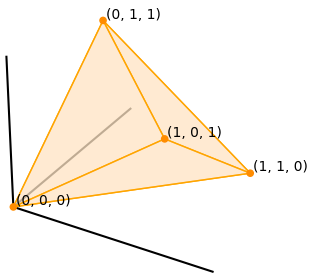}
    \caption{Free join $P_{S_3} \ojoin \{\mathbf{0}\}$.}
    \end{subfigure}
    \caption{Path polytope of a star tree and its free join with the origin.
    }
    \label{fig:plots_polytopes}
\end{figure}
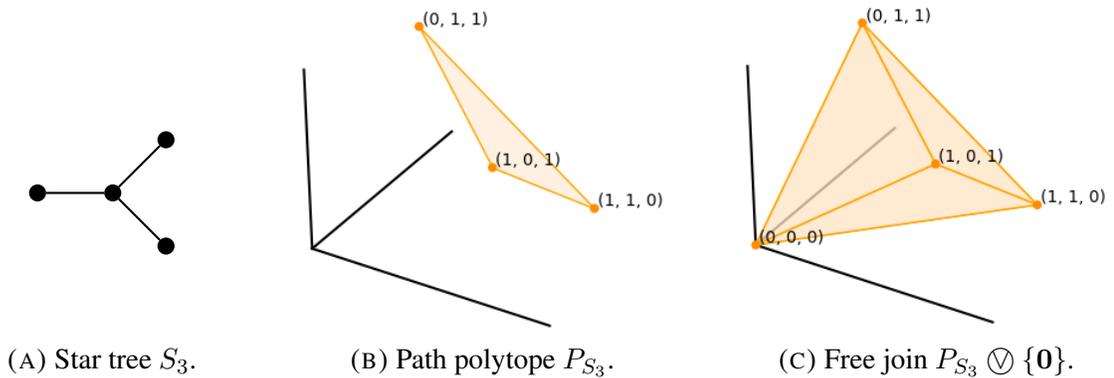

The following result describes the faces of the free join of two polytopes, which is needed to compute facets of path polytopes inductively.

\begin{lemma}[{\cite[Proposition 2.1]{mcmullen1975}}]\label{lemma:facets-ojoin}
The faces of \(P \mathop{\ojoin} Q\) are precisely the sets of the form \(F \mathop{\ojoin} G\), where \(F\) is a face of \(P\) and \(G\) is a face of \(Q\) (including \(F = \emptyset\) or \(P\), and \(G = \emptyset\) or \(Q\)).
\end{lemma}

\subsection{Toric fiber products of polytopes} The notion of toric fiber product was introduced by Sullivant in \cite{sullivant2006toric} on toric ideals. We follow the notation from \cite{hollering2024hyperplane}, which adapts this concept to polytopes.

An \emph{integral} polytope is a polytope whose vertices all have integer coordinates. Given an integral polytope $P \subset \R^n$, a projection $\pi: P \to \R^m$ is called \emph{integral} if $\pi(P)$ is an integral polytope.
Given integral polytopes $P_1$ and $P_2$ and integral projections~$\pi_1: P_1 \rightarrow  Q$ and $\pi_2: P_2 \rightarrow Q$ for some polytope $Q = \pi_1(P_1) = \pi_2(P_2)$, the \textit{toric fiber product} of $P_1$ and $P_2$ is 
\[
P_1 \times_Q P_2 = \conv(\{(\bx,\by) \in \V(P_1) \times \V(P_2) \mid \pi_1(\bx)=\pi_2(\by) \}).
\]
This gives the toric fiber product polytope $P_1 \times_Q P_2$ by its $\mathcal{V}$-representation. We  can retrieve the facets $P_1 \times_Q P_2$ when $Q$ is a standard simplex using the following result.

\begin{lemma}[{\cite[Lemma 3.2]{dinu2021gorenstein}}]
    \label{lemma:facets-tfp}
    Let \(P_1\) and \(P_2\) be two polytopes. Let \(\pi_i : P_i \to \mathbb{R}^n\),  for i=1,2, be integral projections such that \(\pi_1(P_1) = \pi_2(P_2) = \Delta_{n}\). Then all facets of the toric fiber product \(P_1 \times_{\Delta_{n}} P_2\) are of the form \(F_1 \times_{\Delta_{n}} P_2\) or \(P_1 \times_{\Delta_{n}} F_2\), where \(F_i\) is a facet of \(P_i\).
\end{lemma}

\section{Path polytopes of trees via toric fiber products} \label{sec: tfp}

We use toric fiber products to construct the path polytope of a tree using path polytopes of its subtrees. More precisely, just as a $T$ can be formed using gluing on stars $S_{n_1},\ldots,S_{n_k}$, we show that the polytope $P_T$ can be obtained by toric fiber products on $P_{S_{n_1}}\ojoin\{\mathbf{0}\},\ldots,P_{S_{n_k}}\ojoin\{\mathbf{0}\}$. This is done by defining integral projections $\pi_1$ and $\pi_2$ to align corresponding indicator vectors. The resulting matched vectors describe connected paths in the new tree, forming a larger indicator vector for the glued structure. However, some paths between leaves of one subtree need not be glued with paths from different subtrees, requiring the inclusion of the origin in the toric fiber product.

Consider two trees $T_1$ and $T_2$, and two edges $e_1 \in E_\leaf(T_1)$ and $e_2 \in E_\leaf(T_2)$. We call \emph{gluing integral projections for $(T_1, T_2, e_1, e_2)$} to a pair of integral projections $\pi_i : P_{T_i} \ojoin \{ \mathbf{0}\} \to \Delta_{3}$ $(i = 1,2)$ such that $\pi_1(\mathbf{0}) = (0,0,1)$, $ \pi_2(\mathbf{0}) = (1,0,0)$,

\begin{center}
    $\pi_1(\bc_{T_1}^{i \leftrightarrow j}) = \begin{cases}
        (1,0,0) & \text{ if } e_1 \notin E(i \leftrightarrow j) \\
        (0,1,0) & \text{ if } e_1 \in E(i \leftrightarrow j),
    \end{cases}
    $  \quad $\pi_2(\bc_{T_2}^{ i \leftrightarrow j}) = \begin{cases}
        (0,1,0) & \text{ if } e_2 \in E(i \leftrightarrow j) \\
        (0,0,1) & \text{ if }  e_2 \notin E(i \leftrightarrow j).
    \end{cases}$
\end{center}

\begin{example} Consider two trees $T_1$ and $T_2$, and two edges $e_1 \in E_\leaf(T_1)$ and $e_2 \in E_\leaf(T_2)$.  Let $\{\bb_{e}^{(i)} \mid e \in E(T_i) \}$ denote the standard basis of $\R^{E(T_i)}$ for $i=1,2$. Let $\pi_i: \R^{E(T_i)} \to \R^3$ for $i=1,2$ be defined by
    \[
    \pi_1(\bb_e^{(1)}) = \begin{pmatrix}
        \frac{1}{2}\one_{E_\leaf(T_1)}(e) - \one_{\{e_1\}}(e) \\
        \one_{\{e_1\}}(e) \\
        -\frac{1}{2}\one_{E_\leaf(T_1)}(e) + 1
    \end{pmatrix}, \quad \pi_2(\bb_e^{(2)}) = \begin{pmatrix}
    -\frac{1}{2}\one_{E_\leaf(T_2)}(e) + 1\\
        \one_{\{e_2\}}(e) \\
        \frac{1}{2}\one_{E_\leaf(T_2)}(e) - \one_{\{e_2\}}(e)
\end{pmatrix},
\]
where $\mathbbm{1}_S(x)$ denotes the characteristic function for the set $S$. Then $\pi_1,\pi_2$ is a pair of gluing integral projections.
\end{example}

\begin{example}
\label{ex: tfp on two star trees}
 Let $T_1=S_3$ with center $1$ and leaves $2,3,4$, let $T_2=S_3$ with center 5 and leaves $6,7,8$, and let $T=T_1 \ast_{\{1,4\},\{8,5\}} T_2$, as in \Cref{fig:example-toric fiber product}. Let $\pi_1, \pi_2$ be gluing integral projections and let $Q=({P}_{T_1} \ojoin \{\mathbf{0}\})\times_{\Delta_{3}} ({P}_{T_2} \ojoin \{\mathbf{0}\})$. The tables below show the $\V$-representation for each of the polytopes involved, revealing that $P_T \cong Q$ (by mapping the edges $\{1,4\}, \{8,5\}$ to $\{1,5\}$ and leaving the rest fixed). We use colors to denote which vectors are matched together by the gluing integral projections, using the color code given by $\Delta_{3}=\conv(\textcolor{darkgreen}{(1,0,0)},\textcolor{darkred}{(0,1,0)},\textcolor{darkblue}{(0,0,1)})$.

\begin{flushleft}
    \hspace{0.2cm}
    \begin{minipage}{0.4\textwidth}
    \medskip
    \centering
    \renewcommand{\arraystretch}{1.2}
    \setlength{\arrayrulewidth}{0.4mm}
    \begin{tabular}{c|cccc}
        $\V({P}_{T_1} \ojoin \{\mathbf{0}\})$ & $\textcolor{darkgreen}{\bc^{2 \leftrightarrow 3}}$ & $\textcolor{darkred}{\bc^{2 \leftrightarrow 4}}$ & $\textcolor{darkred}{\bc^{3 \leftrightarrow 4}}$ & \textcolor{darkblue}{$\mathbf{0}$} \\
        \hline
        \{1,2\} & 1 & 1 & 0 & 0 \\
        \{1,3\} & 1 & 0 & 1 & 0 \\
        \{1,4\} & 0 & 1 & 1 & 0
    \end{tabular}
    
\bigskip

    \centering
    \renewcommand{\arraystretch}{1.2}
    \setlength{\arrayrulewidth}{0.4mm}
    \begin{tabular}{c|cccc}
        $\V({P}_{T_2} \ojoin \{\mathbf{0}\})$ & $\textcolor{darkblue}{\bc^{6 \leftrightarrow 7}}$ & $\textcolor{darkred}{\bc^{8 \leftrightarrow 6}}$ & $\textcolor{darkred}{\bc^{8 \leftrightarrow 7}}$ & $\textcolor{darkgreen}{\mathbf{0}}$ \\
        \hline
        \{5,8\} & 0 & 1 & 1 & 0 \\
        \{5,6\} & 1 & 1 & 0 & 0 \\
        \{5,7\} & 1 & 0 & 1 & 0
    \end{tabular}
    \end{minipage}
    \hspace{0.05\textwidth}
    \begin{minipage}{0.4\textwidth}
        \renewcommand{\arraystretch}{1.5}

        \setlength{\arrayrulewidth}{0.4mm}
        \centering
        \begin{tabular}{c|cccccc}
     $\mathcal{V}(Q)$ & $\textcolor{darkgreen}{\bc^{2 \leftrightarrow 3}}$ & $\textcolor{darkblue}{\bc^{6 \leftrightarrow 7}}$ & $\textcolor{darkred}{\bc^{2 \leftrightarrow 6}}$ & $\textcolor{darkred}{\bc^{2 \leftrightarrow 7}}$ & $\textcolor{darkred}{\bc^{3 \leftrightarrow 6}}$ & $\textcolor{darkred}{\bc^{3 \leftrightarrow 7 }}$ \\
      \hline
        \{1,2\} & 1 & 0 & 1 & 1 & 0 & 0 \\
        \{1,3\} & 1 & 0 & 0 & 0 & 1 & 1 \\
        \{1,4\} & 0 & 0 & 1 & 1 & 1 & 1 \\
        \{5,8\} & 0 & 0 & 1 & 1 & 1 & 1 \\
        \{5,6\} & 0 &  1 & 1 & 0 & 1 & 0\\
        \{5,7\} & 0 & 1&  0 & 1 & 0 & 1
    \end{tabular}
    \end{minipage}
\end{flushleft}

\begin{center}
    \renewcommand{\arraystretch}{1.5}
    \setlength{\arrayrulewidth}{0.4mm}
   \begin{tabular}{c|cccccc}
     $\mathcal{V}(P_T)$ & $\textcolor{darkgreen}{\bc^{2 \leftrightarrow 3}}$ & $\textcolor{darkblue}{\bc^{6 \leftrightarrow 7}}$ & $\textcolor{darkred}{\bc^{2 \leftrightarrow 6}}$ & $\textcolor{darkred}{\bc^{2 \leftrightarrow 7}}$ & $\textcolor{darkred}{\bc^{3 \leftrightarrow 6}}$ & $\textcolor{darkred}{\bc^{3 \leftrightarrow 7 }}$ \\
      \hline
        \{1,2\} & 1 & 0 & 1 & 1 & 0 & 0 \\
        \{1,3\} & 1 & 0 & 0 & 0 & 1 & 1 \\
        \{1,5\} & 0 & 0 & 1 & 1 & 1 & 1 \\
        \{5,6\} & 0 &  1 & 1 & 0 & 1 & 0\\
        \{5,7\} & 0 & 1&  0 & 1 & 0 & 1
    \end{tabular}
\end{center}
\end{example}

\begin{theorem}
\label{thm:gluing-tfp}

Consider a gluing of two trees $T = T_1 \ast_{e_1, e_2} T_2$. Let $\pi_i : P_{T_i} \ojoin \{ \mathbf{0}\} \to \Delta_{3}$ $(i = 1,2)$ be a pair of gluing integral projections for $(T_1, T_2, e_1, e_2)$. Then,
\[
P_T \cong ({P}_{T_1} \ojoin \{\mathbf{0}\})\times_{\Delta_{3}} ({P}_{T_2} \ojoin \{\mathbf{0}\}).
\]
  \end{theorem}

\begin{proof}
Let $Q = ({P}_{T_1} \ojoin \{\mathbf{0}\})\times_{\Delta_{3}} ({P}_{T_2} \ojoin \{\mathbf{0}\})$. It suffices to construct an affine transformation~$\phi$ that gives a bijection between the vertices of $Q$ and the vertices of $P_T$. Let $e_1 = \{u_1, k_1\}, e_2 = \{u_2, k_2\}$ where $k_i \in \Lv(T_i)$ for $i=1,2$. 
Let $\{\ba_{e} \mid e \in E(T_1) \cup E(T_2)\}$ be the standard basis of $\R^{E(T_1)} \times \R^{E(T_2)} \cong \R^{E(T_1) \cup E(T_2)}$, and let $\{\bb_e \mid e \in E(T) \}$ be the standard basis of $\R^{E(T)}$. Consider the affine map given by 
$\phi(\ba_{e_1}) = \phi(\ba_{e_2})  = \frac{1}{2}\bb_{\{u_1, u_2 \}}$ and
$\phi(\ba_e) = \bb_e$ if $e \neq e_1, e_2$.

The vertices of $Q$ can be divided in three classes, one for each vertex of $\Delta_{3}$.
First, given two distinct leaves $i, j \in \Lv(T_1) \setminus \{ k_1\}$, the vertex $(\bc_{T_1}^{i \leftrightarrow j}, \mathbf{0}) \in \V(Q)$ is mapped to $\bc_T^{i\leftrightarrow j} \in \V(P_T)$ under~$\phi$. 
Second, given two distinct leaves $i, j \in \Lv(T_2) \setminus \{ k_2\}$, the vertex $(\mathbf{0}, \bc_{T_1}^{i \leftrightarrow j}) \in \V(Q)$ is mapped to $\bc_{T}^{i\leftrightarrow j} \in \V(P_T)$ under $\phi$. 
Finally, given $i \in \Lv(T_1) \setminus \{ k_1\}$ and $j \in \Lv(T_2) \setminus \{ k_2\}$ the vertex $(\bc_{T_1}^{i \leftrightarrow k_1}, \bc_{T_2}^{ k_2 \leftrightarrow j}) \in \V(Q)$ is mapped to $\bc_{T}^{i \leftrightarrow j} \in \V(P_T)$ under~$\phi$. 
We have considered all the vertices of both $Q$ and $P_T$, so $P_T \cong Q$.
\end{proof}

\section{Proof of our main theorem} \label{sec: facets for all trees}

In this section, we prove \Cref{thm:main}. We use the decomposition from \Cref{thm:gluing-tfp} together with \Cref{lemma:facets-ojoin} and \Cref{lemma:facets-tfp} to obtain explicit descriptions for the facets of $P_T$. \Cref{lemma: trees decompose into stars} show that every tree can be decomposed into star trees. The following result shows that \Cref{thm:main} is true for star trees, which will serve as a base case of our inductive argument.

\begin{lemma} \label{lemma:halfspaces-Sn}
    Let $S_n = (V, E)$ be the star tree on $n\geq 3$ leaves, and let $u$ be the only internal node of $S_n$, that is, $N(u) = \Lv(S_n)$. Then, $P_{S_n} = \Delta_{n,2}$. Hence, if $n\geq4$ an $\mathcal{H}$-representation of $P_{S_n}$ is given by 
    \[
    \begin{cases}
        x_{\{u, v\}} & \geq 0 \quad \text{for all $ \{u, v\} \in E$} \\
        -x_{\{u,v\}} + \sum\limits_{w \in N(u) \setminus \{v\}} x_{\{u,w\}} &\geq 0 \quad \text{for all $v \in N(u)$}\\
        \sum_{e \in E} x_e & = 2.
    \end{cases}
    \]
    The corresponding set of facets is $\mathcal{F} = \{ F^{\{u,v\}} \mid \{u,v\} \in E\} \cup \{G^{(u,v)} \mid v \in N(u)\}$ where
    \begin{align*}
        F^{\{u,v\}} &= \conv \left(\{ \bc^{i \leftrightarrow j} \mid i,j \in \Lv(S_n), \{u,v\} \not\in i \leftrightarrow j\} \right) \\
        G^{(u,v)} &= \conv \left( \{ \bc^{i \leftrightarrow v } \mid i \in \Lv(S_n) \setminus \{ v\} \}\right).
    \end{align*}
When $n=3$, the first class of halfspaces is redundant and the sets $F^{\{u,v\}}$ are not facets, but the rest of the statement holds.
\end{lemma}

\begin{proof}
    By construction, we have
    \[
    P_{S_n} = \conv \left(\left\{ (x_{\{u, v_1\}}, \dots, x_{\{u, v_n\}}) \in \{0,1\}^n \mid \sum_{e \in E} x_e = 2 \right\} \right)= \Delta_{n,2}.
    \]
    The case $n=3$ is shown in \Cref{fig:example-graph-tree-path}. Assume that $n\geq4$. For a star tree, $E_{\leaf}(T) = E$, so $\sum_{e \in E} x_e = 2$, by \Cref{prop: big hyperplane}. By definition of $\Delta_{n,2}$, the halfspaces of $P_{S_n}$ are $0 \leq x_e \leq 1$ for all $e \in E$. Subtracting $2x_e \leq 2$ from $\sum_{e \in E} x_e = 2$ we get the desired $\mathcal{H}$-representation. Finally, given any $v \in \Lv(S_n)$, we have $F^{\{u,v\}} = P_{S_n} \cap \{ x_e = 0\}$,  $G^{(u,v)} = P_{S_n} \cap \{ x_e = 1\}$, and $\dim(F^{\{u,v\}}) = \dim(G^{(u,v)}) = n-2 = \dim(P_{S_n}) -1$, so the statement follows.
\end{proof}

We are now ready to compute the dimension of path polytopes. \Cref{prop: big hyperplane} implies that $P_T$ has codimension at least one. The next result implies that the codimension of $P_T$ is one provided that $T$ has no internal nodes of degree two.

\begin{proposition}
    \label{prop:dimension-polytope}
    Given a tree $T = (V,E)$ with $|V| > 2$, the dimension of the path polytope $P_T$ is~$|E| -  |\{u \in V \mid \deg(u) = 2\}| - 1$, and $P_T$ is contained in the linear space defined by
    \[
    \begin{cases}
        x_{\{u,v\}} - x_{\{u,w\}} &= 0  \text{\quad for all } u \in \Int(V) \text{ such that } N(u) = \{v,w\}\\
        \sum_{e \in E_{\leaf}(T)} x_{e} &= 2.
    \end{cases}
    \]
    \end{proposition}
\begin{proof}

Let $u\in V$ such that $N(u) = \{v,w \}$. For any pair of distinct leaves $i,j \in \Lv(T)$, we have $\{u, v\} \in i \leftrightarrow j$ if and only if $\{u, w\} \in i \leftrightarrow j$. Hence, every point in the polytope satisfies the first set of equations. The polytope also satisfies the last equation by \Cref{prop: big hyperplane}. These equations are linearly independent, so the dimension of the polytope is at most
\[
|E|  - | \{ v \in V \mid \deg(v) = 2\}| -1.
\]
We show that it is equal to this quantity. Consider a gluing of two trees $T = T_1 \ast_{e_1, e_2} T_2$. Let $\pi_i : P_{T_i} \ojoin \{ \mathbf{0}\} \to \Delta_{3}$ $(i = 1,2)$ be a pair of gluing integral projections for $(T_1, T_2, e_1, e_2)$. Then
\begin{align*}
\dim(P_{T}) &= \dim((P_{T_1} \ojoin \{\mathbf{0}\}) \times_{\Delta_{3}} (P_{T_1} \ojoin \{\mathbf{0}\})) =\\
    &= \dim(P_{T_1} \ojoin \{\mathbf{0}\}) + (P_{T_2} \ojoin \{\mathbf{0}\}) - \dim(\Delta_{3})  =  \\
    & =(\dim(P_{T_1}) + 1) +  (\dim(P_{T_2}) + 1) - 2 =\\
&= \dim(P_{T_1}) + \dim(P_{T_2}).
\end{align*}
Hence, if $T = S_{n_1} \ast_{e_1,e_2} \cdots \ast_{e_{r-1}, e_r} S_{n_r}$ where $S_{n_i}$ is a star tree, then $\dim(P_T) = \sum_{i=1}^r \dim(P_{S_{n_i}})$ and $|E| = 1 + \sum_{i=1}^r (n_i-1)$. Finally, $\dim(P_{S_{2}}) = 0$ and $\dim(P_{S_{n}}) = n-1$ for $n>2$ by \Cref{lemma:halfspaces-Sn}, so the statement follows. 
\end{proof}

\Cref{thm:main} omits the case when $T$ has internal nodes of degree~$2$. This is because the polytope $P_T$ is isomorphic $P_{T'}$ where $T'$ is formed from $T$ by merging  edges $\{u,v\}$ and $\{u,w\}$ to the edge $\{v,w\}$ for any vertex $u$ of degree $2$. Both trees have the same set of leaves. This follows from the following result.

\begin{proposition}
If $T=T'\ast_{e_{1},e_2} S_2$, then $P_T \cong P_{T'}$.
\label{prop:gluing-s2}
\end{proposition}
\begin{proof}
    Let $f \in E(T)$ be the edge resulting from identifying $e_1$ with $e_2$, and let $g \in E(T)$ be the edge in $S_2$ distinct from $e_2$. Let $\{\ba_e \mid e \in E(T')\}$ be the standard basis of $\R^{E(T')}$, and let $\{\bb_e \mid e \in E(T)\}$ be the standard basis of $\R^{E(T)}$. Consider the affine map $\phi:\R^{E(T')} \to \R^{E(T)}$ given by $\phi(\ba_e) = \bb_e$ if $e \neq e_1$ and $\phi(\ba_{e_1}) = \bb_f + \bb_g$. The map $\phi$ is injective and maps the vertices of $P_{T'}$ to the vertices of $P_T$, so $P_T \cong P_{T'}$.
\end{proof}

Propositions \ref{prop:dimension-polytope} and \ref{prop:gluing-s2} imply that we can assume without loss of generality that our trees do not have internal nodes of degree two, so that the corresponding path polytope has codimension one. 

\begin{theorem}\label{thm:facets of trees}
Let $T=(V,E)$ be a tree with $|E| \geq 3$ and with no internal nodes of degree 2. Then, the  facets of $P_T$ are 
\[\{F^{\{u,v\}}_T \mid \{u,v\} \in E, \deg(u), \deg(v) \neq 3\} \cup \{G^{(u,v)}_T \mid u\in \Int(T), v\in N_T(u)\}, \]
where
\begin{align*}
F_{T}^{\{u,v\}} &=\conv \big( \{\bc^{i \leftrightarrow j } \mid i,j \in \Lv(T), \{u,v \} \not\in i\leftrightarrow j \}\big) = 
 P_T \cap \{ x \in \R^{E(T)} \mid x_{\{u,v\}} = 0\}, \text{ and} \\
G_T^{(u,v)} &= \conv \big( \{\bc^{i \leftrightarrow j} \mid  i,j \in \Lv(T), \{u,v\} \in i \leftrightarrow j \text{ or } \{u,w\} \notin i \leftrightarrow j \text{ for any } w \in N_T(u) \} \big) = \\
&= P_T \cap \{ x \in \R^{E(T)} \mid -x_{\{u,v\}} + \sum_{w \in N(u) \setminus \{v\}} x_{\{u,w\}}= 0\}.
\end{align*}
\end{theorem}

\begin{proof}
We prove it by strong induction on $r = |\Int(T)|$. 
The case $r=1$ follows from \Cref{lemma:halfspaces-Sn}. Suppose the claim holds for all trees with at most $r-1$ internal nodes and let $T$ have $r$ internal nodes.
Let $T = T_1 \ast_{e_1, e_2} T_2$, where $T_1$ and $T_2$ are trees with fewer than $r$ internal nodes (see \Cref{lemma: trees decompose into stars}).
Let $e_1 = \{u_1, k_1\}$ and $e_2 = \{u_2, k_2\}$ with $k_1 \in \Lv(T_1)$ and $k_2 \in \Lv(T_2)$.
For simplicity of notation, let $\hat{P}_{T_i} = P_{T_i} \ojoin \{\mathbf{0}\}$ for $i=1,2$. Hence, $P_T \cong \hat{P}_{T_1} \times_{\Delta_{3}} \hat{P}_{T_2}$ by \Cref{thm:gluing-tfp}. 
The facets of $P_{T}$ are isomorphic to facets of the form $\{F_1 \times_{\Delta_{3}} \hat{P}_{T_2}  \mid  F_1 \text{ facet of } \hat{P}_{T_1}\}$ and $\{\hat{P}_{T_1} \times_{\Delta_{3}} F_2 \mid  F_2 \text{ facet of } \hat{P}_{T_2}\}$ by \Cref{lemma:facets-tfp}.
By \Cref{lemma:facets-ojoin} and the induction hypothesis, the facets of $\hat{P}_{T_i}$ for $i=1,2$ are 
\begin{enumerate}
    \item $F_{T_i}^{\{u,v \}} \ojoin \{\mathbf{0}\}$ for $\{u,v\} \in E(T_i)$ with $\deg(u), \deg(v) \neq 3$,
    \item $G_{T_i}^{(u,v)} \ojoin \{\mathbf{0}\}$ for $u \in \Int(T_i)$ and $v$ in the neighborhood $N_{T_i}(v)$ of $v$ in $T_i$, and
    \item $P_{T_i}$.
\end{enumerate}
    
We focus on the facets of the form $\{F_1 \times_{\Delta_{3}} \hat{P}_{T_2}  \mid  F_1 \text{ facet of } \hat{P}_{T_1}\}$. By symmetry, the same reasoning holds for the facets $\{\hat{P}_{T_1} \times_{\Delta_{3}} F_2 \mid  F_2 \text{ facet of } \hat{P}_{T_2}\}$.

\underline{Case 1:} Consider sets of the form $(F_{T_1}^{\{u,v\}} \ojoin \{\mathbf{0}\})$ for $\{u,v\} \in E(T_1)$. First, let $\{u,v\}  = e_1 = \{u_1, k_1\} \in E(T_1)$. Then,
\begin{align*}
    (F^{\{u_1,k_1\}}_{T_1} \ojoin \{ \mathbf{0}\})\times_{\Delta_{3}} \hat{P}_{T_2} = \conv \big(& ({F}_{T_1}^{\{u_1,k_1\}} \times \{\mathbf{0}\}) \cup (\{\mathbf{0}\} \times {F}_{T_2}^{\{u_2,k_2\}} )\big),
\end{align*}
which is isomorphic to $F_T^{\{u_1, u_2\}}$. Note that
\begin{align*}
    &F_{T_1}^{\{u_1, k_1\}} \times \{ \mathbf{0}\} \subset L_1 \coloneqq \left\{ x \in \R^{E(T_1) \cup E(T_2)} \mid x_e = 0 \text{ for all } e \in E(T_2), \sum_{e \in E_\leaf(T_1) \setminus \{e_1\}} x_e = 2 \right\}, \\
    & \{ \mathbf{0}\} \times F_{T_2}^{\{u_2, k_2\}}\subset L_2 \coloneqq \left\{ x \in \R^{E(T_1) \cup E(T_2)} \mid x_e = 0 \text{ for all } e \in E(T_1), \sum_{e \in E_\leaf(T_2) \setminus \{e_2\}} x_e = 2 \right\}.
\end{align*}
Since $L_1$ and $L_2$ are skew linear spaces, we have the free join $({F}_{T_1}^{\{u_1,k_1\}} \times \{\mathbf{0}\}) \ojoin (\{\mathbf{0}\} \times {F}_{T_2}^{\{u_2,k_2\}})$.
By induction hypothesis, $F_{T_i}^{\{u_i, k_i\}}$ is a facet of $P_{T_i}$  ($\dim(F_{T_i}^{\{u_i, k_i\}}) = \dim(P_{T_1}) -1= E(T_i) - 2$) if and only if $\deg(u_i) > 3$, for $i=1,2$. Hence,
\begin{align*}
    \dim(F_T^{\{u_1, u_2\}}) &= \dim(F_{T_1}^{\{u_1, k_1\}}) + \dim(F_{T_2}^{\{u_2, k_2\}}) + 1 \leq (|E(T_1)| - 2) + (|E(T_2)| - 2) - 1 = \\
    &= |E(T)| - 2 = \dim(P_T)-1
\end{align*}
with equality if and only if $\deg(u_1), \deg(u_2) > 3$. So $F_T^{\{u_1, u_2\}} = P_T \cap \{x \in \R^{E(T)} \mid x_{\{u_1, u_2\}} = 0 \}$ is a facet of $P_T$ if and only if $\deg(u_1), \deg(u_2) > 3$.

Next, let $\{ u,v\} \in E(T_1) \setminus \{ e_1\}$ such that $\deg(u), \deg(v) \neq 3$. Assume that $u \in \Int(T_1)$, so $\deg(u) \geq 4$. We have
\begin{align*}
    (F^{\{u,v\}}_{T_1} \ojoin \{ \mathbf{0}\})\times_{\Delta_{3}} \hat{P}_{T_2} = \conv \big(&\{(\bc_{T_1}^{i \leftrightarrow k_1},\bc_{T_2}^{k_2 \leftrightarrow j}) \mid  \bc_{T_1}^{i \leftrightarrow k_1} \in {F}_{T_1}^{\{u,v\}} ,\bc_{T_2}^{k_2 \leftrightarrow j} \in {P}_{T_2}\} \\
    &\cup \{(\bc_{T_1}^{i \leftrightarrow j}, \mathbf{0}) \mid \bc_{T_1}^{i \leftrightarrow j} \in {F}_{T_1}^{\{u,v\}} , i,j \neq k_1\} \\
    &\cup \{ (\mathbf{0}, \bc_{T_2}^{i \leftrightarrow j}) \mid \bc_{T_2}^{i \leftrightarrow j} \in {P}_{T_2}, i,j \neq k_2 \} \big),
\end{align*}
which is isomorphic to $F_{T}^{\{u,v\}}$. Fix two leaves $i,j \in \Lv(T)$ such that $\bc_{T}^{i \leftrightarrow j} \not\in F_{T}^{\{u,v\}}$, that is, $\{u,v\} \in i \leftrightarrow j$. Note that $F_T^{\{u,v\}} \subset \{ x \in \R^{E(T)} \mid x_{\{u,v\}} = 0\} \not\ni \bc_{T}^{i \leftrightarrow j}$. Hence, we have the free join $F_T^{\{u,v\}} \ojoin \bc_{T}^{i \leftrightarrow j}$. We show that $\aff(P_T) = \aff(F_T^{\{u,v\}} \ojoin \bc_{T}^{i \leftrightarrow j})$. Indeed, consider any two leaves $i', j'$ such that $\bc_{T}^{i \leftrightarrow j} \in P_T \setminus F_T^{\{u,v\}}$. Without loss of generality, suppose that $i \neq i'$, $\{u,v\} \not\in i \leftrightarrow i'$,  and $\{u,v\} \in i \leftrightarrow v$ (permute the roles of $i,j$ and $i', j'$ if needed).
Since $\deg(u) > 3$, there exists a leaf $k \in \Lv(T) \setminus \{i, i'\}$ such that $\bc_{T}^{i \leftrightarrow k}, \bc_{T}^{i' \leftrightarrow k} \in F_T^{\{u,v\}}$. We need to consider two cases. If $\deg(v) = 1$, then $j = j' = v$ and $\bc_{T}^{i' \leftrightarrow j} = \bc_{T}^{i \leftrightarrow j} + \bc_{T}^{i' \leftrightarrow k} - \bc_{T}^{ i \leftrightarrow k} \in \aff(F_T^{\{u,v\}} \ojoin \bc_{T}^{ i \leftrightarrow j})$. If $\deg(v) > 3$, then there exists a leaf $l \in \Lv(T) \setminus \{j, j'\}$ such that $\bc_{T}^{ j \leftrightarrow l}, \bc_{T}^{ j' \leftrightarrow l} \in F_T^{\{u,v\}}$, so
\[
\bc_{T}^{ i' \leftrightarrow j'} = \bc_{T}^{ i \leftrightarrow j} + \bc_{T}^{ i' \leftrightarrow k} - \bc_{T}^{ i \leftrightarrow k} + \bc_{T}^{ j' \leftrightarrow l} - \bc_{T}^{ j \leftrightarrow l} \in \aff(F_T^{\{u,v\}} \ojoin \bc_{T}^{ i \leftrightarrow j}).
\]
We have seen that every vertex of $P_T$ lies in $\aff(F_T^{\{u,v\}} \ojoin \bc_{T}^{ i \leftrightarrow j})$, so $\aff(P_T) = \aff(F_T^{\{u,v\}} \ojoin \bc_{T}^{ i \leftrightarrow j})$. Therefore, $\dim(F_T^{\{u,v\}}) = \dim(P_T) - 1$, so $F_T^{\{u,v\}} = P_T \cap \{x \in \R^{E(T)} \mid x_{\{u,v\}} = 0\}$ is a facet of~$P_T$.

\underline{Case 2:} Consider sets of the form $(G^{(u,v)}_{T_1} \ojoin \{ \mathbf{0}\})\times_{\Delta_{3}} \hat{P}_{T_2}$ where $u \in \Int(T_1)$ and $v \in N_{T_1}(u)$. We have
\begin{align*}
    (G^{(u,v)}_{T_1} \ojoin \{ \mathbf{0}\})\times_{\Delta_{3}} \hat{P}_{T_2} = \conv \big(&\{(\bc_{T_1}^{i \leftrightarrow k_1},\bc_{T_2}^{k_2 \leftrightarrow j}) \mid  \bc_{T_1}^{i \leftrightarrow k_1} \in {G}^{(u,v)}_{T_1} ,\bc_{T_2}^{k_2 \leftrightarrow j} \in {P}_{T_2}\} \\
    &\cup \{(\bc_{T_1}^{i \leftrightarrow j}, \mathbf{0}) \mid \bc_{T_1}^{i \leftrightarrow j} \in {G}^{(u,v)}_{T_1} , i,j \neq k_1\} \\
    &\cup \{ (\mathbf{0}, \bc_{T_2}^{i \leftrightarrow j}) \mid \bc_{T_2}^{i \leftrightarrow j} \in {P}_{T_2}, i,j \neq k_2 \} \big)
\end{align*}
which is isomorphic to $G_{T}^{(u,v)}$ if $(u,v) \neq (u_1, k_1)$, and to $G_{T}^{(u_1,u_2)}$ if $(u,v) = (u_1, k_1)$.
Suppose that $(u,v) \neq (u_1, k_1)$, the same reasoning applies to the other case. Fix two leaves $i,j \in \Lv(T)$ such that $\bc_T^{i \leftrightarrow j } \not\in G_T^{(u,v)}$, so there exist $w_1, w_2 \in N_T(u) \setminus \{v\}$ such that $\{u,w_1\}, \{u, w_2\} \in i \leftrightarrow j$.
Note that $G_{T}^{(u,v)} \subset \{ x \in \R^{E(T)} \mid x_{\{u,v\}} = \sum_{w \in N(v) \setminus \{ v\}} x_{\{u,w\}}\} \not \ni \bc_T^{i \leftrightarrow j}$.
Hence, we have the free join $G_{T}^{(u,v)} \ojoin \{\bc_T^{i \leftrightarrow j}\}$. We show that $\aff(P_T) = \aff(G_{T}^{(u,v)} \ojoin \bc_T^{i \leftrightarrow j})$. Indeed, consider any two leaves $i', j'$ such that $\bc_T^{i' \leftrightarrow j'} \in P_T \setminus G_{T}^{(u,v)}$. Pick a leaf $k \in \Lv(T)$ such that $\{u,v\} \in i \leftrightarrow k$, so $\bc_T^{i \leftrightarrow k} \in G_T^{(u,v)}$. For the same leaf $k$, we also have $\bc_T^{j \leftrightarrow k}, \bc_T^{i' \leftrightarrow k}, \bc_T^{j' \leftrightarrow k} \in G_T^{(u,v)}$, so
\[
\bc_T^{i' \leftrightarrow j'} = \bc_T^{i \leftrightarrow j} + \bc_T^{i' \leftrightarrow k} - \bc_T^{i \leftrightarrow k} + \bc_T^{j' \leftrightarrow k} - \bc_T^{j \leftrightarrow k} \in \aff(G_{T}^{(u,v)} \ojoin \bc_T^{i \leftrightarrow j}).
\]
This implies that $\aff(P_T) = \aff(G_{T}^{(u,v)} \ojoin \bc_T^{i \leftrightarrow j})$. Therefore, $\dim(G_{T}^{(u,v)}) = \dim(P_T) - 1$, so $G_{T}^{(u,v)} = P_T \cap \{ x \in \R^{E(T)} \mid x_{\{u,v\}} = \sum_{w \in N(v) \setminus \{ v\}} x_{\{u,w\}}\}$ is a facet of $P_T$.

\underline{Case 3:} Consider sets of the form $P_{T_1} \times_{\Delta_{3}} \hat{P}_{T_2}$. We have that
\begin{align*}
    P_{T_1} \times_{\Delta_{3}} \hat{P}_{T_2} = \conv\big( &\{(\bc^{T_1, i \leftrightarrow k_1}, \bc^{T_2, k_2 \leftrightarrow j}) \mid i \in \Lv(T_1) \setminus \{k_1 \}, j \in \Lv(T_2) \setminus \{ k_2\} \} \\
    &\cup\{(\bc^{T_1, i \leftrightarrow j}, \mathbf{0}) \mid i,j \in \Lv(T_1) \setminus \{k_1\}\}\big)
\end{align*}
is included in $\hat{P}_{T_1} \times_{\Delta_{3}} (G^{(u_2,k_2)}_{T_2} \ojoin \{ \mathbf{0}\}) \cong G_T^{(u_2,u_1)}$, so it is not a new facet of $P_T$.
\end{proof}

We are now ready to prove \Cref{thm:main}l that is, (\ref{eq:main}) is a  minimal $\mathcal{H}$-representation of the path polytope $P_T$. 

\begin{proof}
[Proof of \Cref{thm:main}]
Consider a tree $T$ with no internal nodes of degree $2$. The last equality in (\ref{eq:main}) comes from \Cref{prop: big hyperplane}. Every vertex $\bc^{i \leftrightarrow j}$ satisfies the inequalities and therefore every point in the polytope also satisfies the inequalities. By \Cref{thm:facets of trees}, every facet of $P_T$ is the intersection of the polytope with a hyperplane corresponding to the boundary of a halfspace in the $\mathcal{H}$-representation given in (\ref{eq:main}). Hence, (\ref{eq:main}) is a minimal $\mathcal{H}$-representation.  Lastly, $\dim(P_T) = |E| -1 $ by \Cref{prop:dimension-polytope}.
\end{proof}

\begin{corollary}
    \label{cor: dimension of polytope}
Given a tree $T = (V,E)$, let $\mathcal{H}$ be the halfspaces given in \Cref{thm:main}. Then, $P_T$ has dimension $|E| - \{u \in V \mid \deg(u)=2\}-1$ and a halfspace description 
\begin{align*}\label{eq:degree2 halfspaces}
 \mathcal{H} \cup \{x_{\{u,v\}} - x_{\{u,w\}} = 0 \mid u\in \Int(T), N(u) = \{v,w\}\}. 
\end{align*}
\end{corollary}

\begin{remark}
    Note that the $\mathcal{H}$-representation from \Cref{cor: dimension of polytope} is not minimal when internal nodes of degree $2$ are present: some  inequalities~$x_{\{v,u\}} \geq 0$ are redundant. For example, let $N(v) = \{ u, w\}$. Then $x_{\{v,u\}} - x_{\{v,w\}} = 0$ and $x_{\{v,u\}} \geq 0$ imply $x_{\{v,w\}} \geq 0$. 
\end{remark}

\begin{remark}
We expect the techniques presented in this paper to be useful to other polytopes defined by suitable ``gluing" integral projections,  provided that the polytopes are contained in a hyperplane that does not contain the origin (using \Cref{lemma:facets-ojoin}).
However, when the path polytope is constructed from all vectors $\bc^{i \leftrightarrow j}$ for any two nodes $i$ and $j$, not only leaves (e.g., see \cite{coons2023generalized}), the resulting polytope is full-dimensional. Hence, in this case one should seek alternative approaches. 
\end{remark}

\section*{Acknowledgments}
AG and AM were supported by the National Science Foundation under Grant DMS-2306672. AG was supported by an REU at the University of Michigan funded under this grant. AR was supported by fellowships from ``la
Caixa” Foundation (ID 100010434), with code LCF/BQ/EU23/12010097, and from RCCHU. Most of the work was done in Pierce Hall (SEAS) at Harvard University. 

\bibliographystyle{plain}
\bibliography{references}

\end{document}